\documentclass[11pt, notitlepage]{article}   % The document class has to be "report"

\usepackage{amsmath,amsthm,amsfonts}   % These are to write mathematics

\usepackage{amscd}
\usepackage{amsfonts}
\usepackage{amssymb}
\usepackage{color}      % Need the color package
\usepackage{epsfig}
\usepackage{graphicx}           % Graphics package

\usepackage{tikz}
\theoremstyle{plain}
\newtheorem{theorem}{Theorem}[section]
\newtheorem{lemma}[theorem]{Lemma}
\newtheorem{corollary}[theorem]{Corollary}
\newtheorem{proposition}[theorem]{Proposition}
\newtheorem{observation}[theorem]{Observation}
\newtheorem{example}[theorem]{Example}
\newtheorem{remark}[theorem]{Remark}
\newtheorem{conjecture}[theorem]{Conjecture}
\newtheorem{question}[theorem]{Question}

\theoremstyle{definition}
\newtheorem{definition}[theorem]{Definition}

\newcommand{\diam}{\textnormal{diam}}
\newcommand{\code}{\textnormal{code}}

\def\finf{\mathop{{\rm I}\kern -.27 em {\rm F}}\nolimits}

\newcommand{\rmv}[1]{}
\begin{document}

\title{Maker-Breaker Metric Resolving Games on Graphs}

\author{{\bf{Cong X. Kang}}$^1$ and {\bf{Eunjeong Yi}}$^2$\\
\small Texas A\&M University at Galveston, Galveston, TX 77553, USA\\
{\small\em kangc@tamug.edu}$^1$; {\small\em yie@tamug.edu}$^2$}

\maketitle

\date{}

\begin{abstract}
 Let $d(x,y)$ denote the length of a shortest path between vertices $x$ and $y$ in a graph $G$ with vertex set $V$. For a positive integer $k$, let $d_k(x,y)=\min\{d(x,y), k+1\}$ and $R_k\{x,y\}=\{z\in V: d_k(x,z) \neq d_k(y,z)\}$. A set $S \subseteq V$ is a \emph{distance-$k$ resolving set} of $G$ if $S \cap R_k\{x,y\} \neq\emptyset$ for distinct $x,y\in V$. In this paper, we study the maker-breaker distance-$k$ resolving game (MB$k$RG) played on a graph $G$ by two players, Maker and Breaker, who alternately select a vertex of $G$ not yet chosen. Maker wins by selecting vertices which form a distance-$k$ resolving set of $G$, whereas Breaker wins by preventing Maker from winning. We denote by $O_{R,k}(G)$ the outcome of MB$k$RG. Let $\mathcal{M}$, $\mathcal{B}$ and $\mathcal{N}$, respectively, denote the outcome for which Maker, Breaker, and the first player has a winning strategy in MB$k$RG. Given a graph $G$, the parameter $O_{R,k}(G)$ is a non-decreasing function of $k$ with codomain $\{-1=\mathcal{B}, 0=\mathcal{N}, 1=\mathcal{M}\}$. We exhibit pairs $G$ and $k$ such that the ordered pair $(O_{R,k}(G), O_{R, k+1}(G))$ realizes each member of the set $\{(\mathcal{B}, \mathcal{N}),(\mathcal{B}, \mathcal{M}),(\mathcal{N},\mathcal{M})\}$; we provide graphs $G$ such that $O_{R,1}(G)=\mathcal{B}$, $O_{R,2}(G)=\mathcal{N}$ and $O_{R,k}(G)=\mathcal{M}$ for $k\ge3$. Moreover, we obtain some general results on MB$k$RG and study the MB$k$RG played on some graph classes.
\end{abstract}

\noindent\small {\bf{Keywords:}} distance-$k$ metric, $k$-truncated metric, resolving set, $k$-truncated resolving set, distance-$k$ resolving set, Maker-Breaker distance-$k$ resolving game, Maker-Breaker resolving game\\

\noindent \small {\bf{2010 Mathematics Subject Classification:}} 05C12, 05C57\\

%%%%%%%%%%%%%%%%%%%%%%%%%%%%%%%
%%%%%%%%%%%%%%%%%%%%%%%%%%%%%%%

\section{Introduction}

Games played on graphs have been studied extensively; examples of two-player games include cop and robber game~\cite{cop}, Hex board game~\cite{hex}, Maker-Breaker domination game~\cite{mbdg}, etc. Erd\"{o}s and Selfridge~\cite{erdos} introduced the Maker-Breaker game played on an arbitrary hypergraph $H=(V,E)$ by two players, Maker and Breaker, who alternately select a vertex from $V$ not yet chosen in the course of the game. Maker wins the game if he can select all vertices of a hyperedge from $E$, whereas Breaker wins if she is able to prevent Maker from doing so. For further reading on these games, see~\cite{beck, hefetz}. The Maker-Breaker resolving game (MBRG) was introduced in~\cite{mbrg}, and its fractionalization was studied in~\cite{FMBRG}. As it turns out, on any graph $G$ that is connected, MBRG fits as the terminal member of a natural family of metric resolving games numbering $(\diam(G)-1)$ in strength, where $\diam(G)$ is the length of a longest path found in $G$. In this paper, we introduce and study this family, and we call a general member of this family the \emph{Maker-Breaker distance-$k$ resolving game} (MB$k$RG). As will be made clear shortly, MB$k$RG equals MBRG when $k=\diam(G)-1$. But first, we need to set down some basic terminology and notations.

Let $G$ be a finite, simple, undirected, and connected graph with vertex set $V(G)$ and edge set $E(G)$. Let $k$ be a member of the set $\mathbb{Z}^+$ of positive integers. For $x,y\in V(G)$, let $d(x,y)$ denote the minimum number of edges in a path linking $x$ and $y$ in $G$, and let $d_k(x,y)=\min\{d(x,y),k+1\}$. Thus, $d(\cdot,\cdot)$ is the usual, shortest-path metric on $G$, and we can call $d_k(\cdot,\cdot)$ the \emph{distance-$k$} or the \emph{$k$-truncated metric} on $G$. For distinct $x,y\in V(G)$, let $R\{x,y\}=\{z\in V(G): d(x,z)\neq d(y,z)\}$ and $R_k\{x,y\}=\{z\in V(G): d_k(x,z)\neq d_k(y,z)\}$. A set $S\subseteq V(G)$ is a \emph{resolving set} of $G$ if $S \cap R\{x,y\} \neq\emptyset$ for distinct $x,y\in V(G)$, and the \emph{metric dimension} $\dim(G)$ of $G$ is the minimum cardinality over all resolving sets of $G$. Similarly, a set $S \subseteq V(G)$ is a \emph{distance-$k$ resolving set} (also called a \emph{$k$-truncated resolving set}) of $G$ if $S \cap R_k\{x,y\} \neq\emptyset$ for distinct $x,y\in V(G)$, and the \emph{distance-$k$ metric dimension} (also called the \emph{$k$-truncated metric dimension}) $\dim_k(G)$ of $G$ is the minimum cardinality over all distance-$k$ resolving sets of $G$. For an ordered set $S=\{u_1, u_2, \ldots, u_{\alpha}\} \subseteq V(G)$ and for a vertex $v\in V(G)$, the metric code of $v$ with respect to $S$ is the $\alpha$-vector $\code_{S}(v)=(d(v, u_1), d(v, u_2), \ldots, d(v, u_{\alpha}))$. We note that $S$ is a resolving set of $G$ if and only if the map $\code_{S}(\cdot)$ is injective on $V(G)$. By replacing the metric $d(\cdot,\cdot)$ by $d_k(\cdot,\cdot)$ mutatis mutandis, the notion of a distance-$k$ metric code map $\code_{S,k}(\cdot)$ on $V(G)$ and  a distance-$k$ resolving set are analogously defined; note that, again, $S\subseteq V(G)$ is a distance-$k$ resolving set of $G$ if and only if the map $\code_{S,k}(\cdot)$ is injective on $V(G)$.

The concept of metric dimension was introduced in~\cite{harary, slater}, and the concept of distance-$k$ metric dimension was introduced in~\cite{beardon, moreno_thesis}. For further study on distance-$k$ metric dimension of graphs, see~\cite{distKdim}, which is a result of merging~\cite{JE} and~\cite{tilquist} with some additional results, and~\cite{tilquist} is based on~\cite{tilquist_thesis}. Some applications of metric dimension include robot navigation~\cite{landmarks}, network discovery and verification~\cite{network}, chemistry~\cite{chemistry}, sonar~\cite{slater} and combinatorial optimization~\cite{sebo}. For applications of distance-$k$ metric dimension see~\cite{distKdim, adim}, where sensors/landmarks are placed at locations (vertices) forming a distance-$k$ resolving set of a network, with the understanding that the cost of a sensor/landmark increases as its detection range increases. It is known that determining the metric dimension of a general graph is an NP-hard problem (\cite{NP, landmarks}) and that determining the distance-$k$ metric dimension of a general graph is an NP-hard problem (\cite{juan1, juan2}).

Returning to the eponymous games of this paper, and following~\cite{mbrg}, MBRG (MB$k$RG, respectively) is played on a graph $G$ by two players, Maker (also called Resolver) and Breaker (also called Spoiler), denoted by $M^*$ and $B^*$, respectively. $M^*$ and $B^*$ alternately select (without missing their turn) a vertex of $G$ that was not yet chosen in the course of the game. $M$-game ($B$-game, respectively) denotes the game for which $M^*$ ($B^*$, respectively) plays first. $M^*$ wins the MBRG (MB$k$RG, respectively) if he is able to select vertices that form a resolving set (a distance-$k$ resolving set, respectively) of $G$ in the course of the game, and $B^*$ wins MBRG (MB$k$RG, respectively) if she stops Maker from winning. We denote by $O_R(G)$ and $O_{R,k}(G)$ the outcomes, respectively, of MBRG and MB$k$RG played on a graph $G$. Noting that there's no advantage for the second player in the MBRG, it was observed in~\cite{mbrg} that there are three realizable outcomes, as follows: $O_R(G)=\mathcal{M}$ if $M^*$ has a winning strategy whether he plays first or second in the MBRG, $O_R(G)=\mathcal{B}$ if $B^*$ has a winning strategy whether she plays first or second in the MBRG, and $O_R(G)=\mathcal{N}$ if the first player has a winning strategy in the MBRG. Analogously, we assign to $O_{R,k}(G)$ an element of $\{\mathcal{M}, \mathcal{B}, \mathcal{N}\}$ in accordance with each of the aforementioned outcomes.

The authors of~\cite{mbrg} studied the minimum number of moves needed for $M^*$ ($B^*$, respectively) to win the MBRG provided $M^*$ ($B^*$, respectively) has a winning strategy. In MBRG, let $M_{R}(G)$ ($M'_R(G)$, respectively) denote the minimum number of moves for $M^*$ to win the $M$-game ($B$-game, respectively) provided he has a winning strategy with $O_R(G)=\mathcal{M}$, and let $B_{R}(G)$ ($B'_R(G)$, respectively) denote the minimum number of moves for $B^*$ to win the $M$-game ($B$-game, respectively) provided she has a winning strategy with $O_R(G)=\mathcal{B}$. Suppose rivals $X$ and $Y$ compete to gain control over a network $Z$, with $X$ trying to install transmitters with limited range at certain nodes to form a distance-$k$ resolving set of the network, while $Y$ seeks to sabotage the effort by $X$. In this scenario, time becomes a matter of natural concern. If $O_{R,k}(G)=\mathcal{M}$, then we denote by $M_{R,k}(G)$ ($M'_{R,k}(G)$, respectively) the minimum number of moves for $M^*$ to win the $M$-game ($B$-game, respectively). If $O_{R,k}(G)=\mathcal{B}$, then we denote by $B_{R,k}(G)$ ($B'_{R,k}(G)$, respectively) the minimum number of moves for $B^*$ to win the $M$-game ($B$-game, respectively). If $O_{R,k}(G)=\mathcal{N}$, then we denote by $N_{R,k}(G)$ ($N'_{R,k}(G)$, respectively) the minimum number of moves for the first player to win the $M$-game ($B$-game, respectively).

Now, we recall a bit more terminology and notations. The \emph{diameter}, $\diam(G)$, of $G$ is $\max\{d(x,y): x,y \in V(G)\}$. For $v\in V(G)$, the \emph{open neighborhood} of $v$ is $N(v)=\{u \in V(G) : uv \in E(G)\}$, and the \emph{degree} of $v$ is $|N(v)|$; a \emph{leaf} (or an \emph{end-vertex}) is a vertex of degree one, and a \emph{major vertex} is a vertex of degree at least three. Let $P_n$, $C_n$ and $K_n$ respectively denote the path, the cycle and the complete graph on $n$ vertices. For positive integers $s$ and $t$, let $K_{s,t}$ denote the complete bi-partite graph with two parts of sizes $s$ and $t$. For $\alpha\in\mathbb{Z}^+$, let $[\alpha]$ denote the set $\{1,2,\ldots, \alpha\}$.

This paper is organized as follows. In Section~\ref{sec_comparison_outcomes}, we study the parameter $O_{R,k}(G)$ as a function of $k$. We exhibit pairs $G$ and $k$ such that the ordered pair $(O_{R,k-1}(G), O_{R,k}(G))$ realizes each member of the set $\{(\mathcal{B},\mathcal{N}),(\mathcal{B},\mathcal{M}),(\mathcal{N},\mathcal{M})\}$, and we define an MB$k$RG-outcome-transition number (a jump) as an integer $k>1$ such that $O_{R, k-1}(G)\neq O_{R,k}(G)$. In Section~\ref{sec_examples}, we examine $O_{R,k}(G)$ when $G$ is the Petersen graph, a complete multipartite graph, a wheel graph, a cycle, and a tree with some restrictions.

%%%%%%%%%%%%%%%%%%%%%%%%%%%%%%%
%%%%%%%%%%%%%%%%%%%%%%%%%%%%%%%

\section{The parameter $O_{R,k}(G)$ as a function of $k$}\label{sec_comparison_outcomes}

To facilitate our discussion of $O_{R,k}(G)$ as a function of $k$, let us assign outcomes $\mathcal{B}, \mathcal{N}, \text{and\ }\mathcal{M}$ the values of $-1, 0, \text{and\ }1$, respectively, and hence we have $\mathcal{B}<\mathcal{N}<\mathcal{M}$. It is clear from the definitions that $d_{k+1}(x,y)\geq d_k(x,y)$ and, for $k\geq \diam(G)-1$, $d_k(x,y)=d(x,y)$. Thus, we can identify $d(x,y)$ with $d_k(x,y)$ for any $k\geq \diam(G)-1$; we can likewise identify $O_R(G)$ with $O_{R,k}(G)$ for any $k\geq \diam(G)-1$. The following observation is instrumental to studying MB$k$RG as a function of $k$.

\begin{observation}\label{monotonicity_1}
On any graph, a distance-$k$ resolving set is a distance-$(k+1)$ resolving set.
\end{observation}

The preceding observation easily leads to the following two observations, which are found in existing literature.

\begin{observation}\emph{\cite{beardon}}\label{obs_dim,dimK}
Let $G$ be a connected graph, and let $k,k'\in\mathbb{Z}^+$ with $k<k'$. Then $\dim_{k}(G)\geq \dim_{k'}(G)\geq\dim(G)$.
\end{observation}

\begin{observation}\emph{\cite{moreno_thesis}}\label{obs_dimK_diameter}
Let $k\in\mathbb{Z}^+$ and $G$ be a connected graph.
\begin{itemize}
\item[(a)] If $\diam(G)\in\{1,2\}$, then $\dim_k(G)=\dim(G)$ for all $k$.
\item[(b)] More generally, $\dim_k(G)=\dim(G)$ for all $k\ge \diam(G)-1$.
\end{itemize}
\end{observation}

The following monotonicity result follows readily from the rules of MB$k$RG and Observation~\ref{monotonicity_1}.

\begin{proposition}\label{monotonicity_2}
On any graph $G$, the parameter $O_{R,k}(G)$ is a non-decreasing function of $k$, where the codomain is $\{\mathcal{B}=-1, \mathcal{N}=0, \mathcal{M}=1\}$.
\end{proposition}

It's worth noting that there are exactly ${\diam(G)+1 \choose 2}$ monotone functions from $[\diam(G)-1]$ to $\{-1,0,1\}$. To see this, observe that such a function is but a vector consisting of a string of $-1$'s, followed by a string of $0$'s, and then followed by a string of $1$'s, satisfying $a+b+c=\diam(G)-1$ and $a,b,c\in\{0\}\cup\mathbb{Z}^+$, where $a$, $b$, and $c$ respectively denote the number of $-1$'s, $0$'s, and $1$'s. The number of solutions to the diophantine equation clearly equals the number of monomials of total degree $(\diam(G)-1)$ in $3$ symbols, and the latter is well known to be the combinatoric symbol asserted. Now, we spotlight a few key results among the bounty of consequences yielded by Proposition~\ref{monotonicity_2}.

\begin{corollary}\label{mbrg_mbtrg}
Let $k\in\mathbb{Z}^+$ and $G$ be a connected graph.
\begin{itemize}
\item[(a)] If $\diam(G)\in\{1,2\}$, then $O_{R,k}(G)=O_R(G)$ for all $k$.
\item[(b)] Given $k<k'$, we have $(O_{R,k}(G), O_{R,k'}(G))\in\{(\mathcal{B},\mathcal{B}),(\mathcal{N},\mathcal{N}),(\mathcal{M},\mathcal{M})$, $(\mathcal{B},\mathcal{N}),(\mathcal{B},\mathcal{M}),(\mathcal{N},\mathcal{M})\}$; note that $O_{R,k_0}(G)$ is $O_{R}(G)$ for $k_0\geq \diam(G)-1$.
\end{itemize}
\end{corollary}

The monotonicity of $O_{R,k}(G)$ as a function of $k$ with codomain $\{\mathcal{B},\mathcal{N},\mathcal{M}\}$ prompts natural questions. Given a graph $G$, what is the range of $O_{R,k}(G)$? Also, where does the function ``jump value"? To this end, for $X, Y\in \{\mathcal{B},\mathcal{N},\mathcal{M}\}$ with $X<Y$, define \emph{the MB$k$RG-outcome-transition number (the jump) of $G$ from $X$ to $Y$}, denoted by $O^T_{X,Y}(G)$, to be the number $\alpha \in [\diam(G)-1]-\{1\}$ satisfying $O_{R,\alpha-1}(G)=X$ and $O_{R,\alpha}(G)=Y$, when such an $\alpha$ exists; put $O^T_{X,Y}(G)=\emptyset$ otherwise. Clearly, each graph $G$ has no more than two jumps and, if $O^T_{\mathcal{B},\mathcal{M}}(G)\neq\emptyset$, then $G$ has only one jump.\\

Now, we work towards realization results in conjunction with Proposition~\ref{monotonicity_2} and Corollary~\ref{mbrg_mbtrg}. First, we recall some terminology. Two vertices $u$ and $w$ are called \emph{twins} if $N(u)-\{w\}=N(w)-\{u\}$; notice that a vertex is its own twin. Hernando et al.~\cite{Hernando} observed that the twin relation is an equivalence relation on $V(G)$ and, under it, each (\emph{twin}) equivalence class induces either a clique or an independent set. The next few results involve twin equivalence classes.

\begin{observation}\label{obs_twin}
Let $x$ and $y$ be distinct members of the same twin equivalence class of $G$.
\begin{itemize}
\item[(a)] \emph{\cite{Hernando}} If $R$ is a resolving set of $G$, then $R \cap\{x,y\} \neq\emptyset$.
\item[(b)] \emph{\cite{distKdim}} If $R_k$ is a distance-$k$ resolving set of $G$, then $R_k\cap \{x,y\} \neq\emptyset$.
\end{itemize}
\end{observation}

\begin{proposition}\label{mbrg_twin}
Let $G$ be a connected graph of order at least $4$.
\begin{itemize}
\item[(a)] \emph{\cite{mbrg}} If $G$ has a twin equivalence class of cardinality at least $4$, then $O_R(G)=\mathcal{B}$.
\item[(b)] \emph{\cite{mbrg}} If $G$ has two distinct twin equivalence classes of cardinality at least $3$, then $O_R(G)=\mathcal{B}$.
\item[(c)] \emph{\cite{FMBRG}} If $G$ has $k \ge 0$ twin equivalence class(es) of cardinality $2$ and exactly one twin equivalence class of cardinality $3$ with $\dim(G)=k+2$, then $O_R(G)=\mathcal{N}$.
\end{itemize}
\end{proposition}

\begin{corollary}\label{mbtrg_twin}
Let $k\in\mathbb{Z}^+$, and let $G$ be a connected graph of order at least $4$.
\begin{itemize}
\item[(a)] If $G$ has a twin equivalence class of cardinality at least $4$, then $O_{R, k}(G)=\mathcal{B}$ for all $k$.
\item[(b)] If $G$ has two distinct twin equivalence classes of cardinality $3$, then $O_{R,k}(G)=\mathcal{B}$ for all $k$.
\end{itemize}
\end{corollary}

Here is a relation between $\dim_k(G)$ and $O_{R,k}(G)$, which is analogous to the relation between $\dim(G)$ and $O_R(G)$ obtained in~\cite{mbrg}.

\begin{observation}\label{obs_dimK}
Let $k\in\mathbb{Z}^+$, and let $G$ be a connected graph of order $n\ge2$.
\begin{itemize}
\item[(a)] If $O_{R,k}(G)=\mathcal{M}$, then $\dim_k(G)\le \lfloor\frac{n}{2}\rfloor$.
\item[(b)] If $\dim_k(G) \ge \lceil\frac{n}{2}\rceil+1$, then $O_{R,k}(G)=\mathcal{B}$.
\end{itemize}
\end{observation}

\begin{proof}
Let $k\in\mathbb{Z}^+$, and let $G$ be a connected graph of order $n\ge2$.

(a) Let $O_{R,k}(G)=\mathcal{M}$. Assume, to the contrary, that $\dim_k(G)> \lfloor\frac{n}{2}\rfloor$. In the $B$-game of the MB$k$RG, $M^*$ can occupy at most $\lfloor\frac{n}{2}\rfloor$ vertices, and hence $M^*$ fails to occupy a distance-$k$ resolving set of $G$; thus, $O_{R,k}(G)\neq\mathcal{M}$, which contradicts the hypothesis.

(b) Let $\dim_k(G)\ge \lceil\frac{n}{2}\rceil+1$. In the $M$-game of the MB$k$RG, $M^*$ can occupy at most $\lceil\frac{n}{2}\rceil$ vertices of $G$; thus, $M^*$ fails to occupy vertices that form a distance-$k$ resolving set of $G$. Since $B^*$ has a winning strategy for the $M$-game in the MB$k$RG, $O_{R,k}(G)=\mathcal{B}$.~\hfill
\end{proof}

Analogous to the concept of a pairing dominating set (see~\cite{mbdg}) and a pairing resolving set (see~\cite{mbrg}), we define a pairing distance-$k$ resolving set and a quasi-pairing distance-$k$ resolving set of a graph.

\begin{definition}
Let $k, \alpha\in\mathbb{Z}^+$ and $G$ be a connected graph. Let $X=\displaystyle\bigcup_{i\in [\alpha]}\{\{u_i, w_i\}\}$, where $\bigcup X\subseteq V(G)$ and $|\bigcup X|=2\alpha$. Let $Z\subseteq V(G)$ be such that $|Z|=\alpha$ and $Z\cap \{u_i, w_i\}\neq \emptyset$ for each $i\in[\alpha]$.
\begin{itemize}
\item[(a)] Suppose each $Z$, as defined, is a distance-$k$ resolving set of $G$, then $X$ is called a \emph{pairing distance-$k$ resolving set} of $G$.
\item[(b)] Suppose each $Z$, as defined, fails to be a distance-$k$ resolving set of $G$, and there exists a vertex $v\in V(G)- \bigcup X$ such that $Z \cup \{v\}$ is a distance-$k$ resolving set of $G$ for each $Z$, then $X$ is called a \emph{quasi-pairing distance-$k$ resolving set} of $G$.
\end{itemize}
\end{definition}

\begin{observation}\emph{\cite{mbrg}}\label{pair_resolving_mbrg}
If $G$ admits a pairing resolving set, then $O_{R}(G)=\mathcal{M}$.
\end{observation}

\begin{observation}\label{pair_resolving}
Let $k\in\mathbb{Z}^+$, and let $G$ be a connected graph of order at least two.
\begin{itemize}
\item[(a)] If $G$ admits a pairing distance-$k$ resolving set, then $O_{R,k}(G)=\mathcal{M}$.
\item[(b)] If $G$ admits a quasi--pairing distance-$k$ resolving set, then $O_{R,k}(G)\in\{\mathcal{M},\mathcal{N}\}$.
\end{itemize}
\end{observation}

Now, we present the main theorem of this section, which is a collection of realization results (to wit, examples) which shed much light on MB$k$RG, viewed as a family of $(\diam(G)-1)$ resolving games played out on a fixed graph $G$; in other words, viewing $O_{R,k}(G)$ as a function of $k$.

\begin{theorem}\label{ex_mbrg,mbtrg}
Let $k\in\mathbb{Z}^+$.
\begin{itemize}
\item[(a)] There exist graphs $G$ satisfying $O_{R,k}(G)=\mathcal{M}$ for $k\ge1$.
\item[(b)] There exist graphs $G$ satisfying $O_{R,k}(G)=\mathcal{N}$ for $k\ge1$.
\item[(c)] There exist graphs $G$ satisfying $O_{R,k}(G)=\mathcal{B}$ for $k\ge1$.
\item[(d)] There exists a graph $G$ such that $O_{R,1}(G)=\mathcal{N}$ and $O_{R,k}(G)=\mathcal{M}$ for $k\ge2$.
\item[(e)] There exist graphs $G$ such that $O_{R,1}(G)=\mathcal{B}$ and $O_{R,k}(G)=\mathcal{N}$ for $k\ge2$.
\item[(f)] There exist graphs $G$ such that $O_{R,1}(G)=\mathcal{B}$ and $O_{R,k}(G)=\mathcal{M}$ for $k\ge2$.
\item[(g)] There exist graphs $G$ such that $O_{R,1}(G)=\mathcal{B}$, $O_{R,2}(G)=\mathcal{N}$ and $O_{R,k}(G)=\mathcal{M}$ for $k\ge3$.
\end{itemize}
\end{theorem}

\begin{proof}
(a) Let $G$ be a tree obtained from $K_{1, \alpha}$, where $\alpha\ge3$, by subdividing exactly $(\alpha-2)$ edges once. Let $\ell_1, \ell_2, \ldots, \ell_{\alpha}$ be the leaves and $v$ be the major vertex of $G$ such that $d(v, \ell_{\alpha-1})=d(v, \ell_{\alpha})=1$ and $d(v, \ell_i)=2$ for each $i\in[\alpha-2]$, and let $s_i$ be the degree-two vertex lying on the $v-\ell_i$ path for each $i\in[\alpha-2]$. Since $\{\{\ell_{\alpha-1}, \ell_{\alpha}\}\}\cup(\cup_{i=1}^{\alpha-2}\{\{s_i, \ell_i\}\})$ is a pairing distance-$k$ resolving set of $G$ for each $k\in\mathbb{Z}^+$, $O_{R,k}(G)=\mathcal{M}$ for all $k\ge1$ by Observation~\ref{pair_resolving}(a).

(b) Let $G$ be a tree obtained from $K_{1,\alpha}$, where $\alpha\ge4$, by subdividing exactly $(\alpha-3)$ edges once. Let $v$ be the major vertex of $G$ and let $\{\ell_1, \ell_2, \ldots, \ell_{\alpha}\}$ be the set of leaves of $G$ such that $d(v, \ell_{\alpha-2})=d(v, \ell_{\alpha-1})=d(v, \ell_{\alpha})=1$ and $d(v, \ell_i)=2$ for each $i\in[\alpha-3]$; further, for each $i\in[\alpha-3]$, let $s_i$ be the degree-two vertex lying on the $v-\ell_i$ path in $G$. We note that, for any distance-$k$ resolving set $R$ of $G$, $|R \cap \{ \ell_{\alpha-2}, \ell_{\alpha-1},\ell_{\alpha}\}|\ge2$ by Observation~\ref{obs_twin}(b). Let $X=\{\{\ell_{\alpha}, \ell_{\alpha-1}\}\} \cup (\cup_{i=1}^{\alpha-3}\{\{s_i, \ell_i\}\})$ and let $Z\subseteq V(G)$ with $|Z|=\alpha-2$ such that $Z \cap \{\ell_{\alpha}, \ell_{\alpha-1}\} \neq\emptyset$ and $Z \cap \{s_i, \ell_i\} \neq\emptyset$ for each $i\in[\alpha-3]$. Since $Z\cup\{\ell_{\alpha-2}\}$ is a distance-$k$ resolving set of $G$ and $Z$ fails to form a distance-$k$ resolving set of $G$, $X$ is a quasi-pairing distance-$k$ resolving set of $G$. By Observation~\ref{pair_resolving}(b), $O_{R,k}(G)\in\{\mathcal{M}, \mathcal{N}\}$. In the $B$-game, $B^*$ can select two vertices in $\{\ell_{\alpha}, \ell_{\alpha-1}, \ell_{\alpha-2}\}$ after her second move; thus, $B^*$ wins the $B$-game. Therefore, $O_{R,k}(G)=\mathcal{N}$ for all $k\ge1$.

(c) Let $G=K_{1,\beta}$, where $\beta\ge4$, be the star on ($\beta+1$) vertices such that $L(G)=\cup_{i=1}^{\beta}\{\ell_i\}$ is the set of leaves of $G$. Since $L(G)$ is a twin equivalence class of cardinality $\beta\ge4$, $O_{R,k}(G)=\mathcal{B}$ for all $k\ge1$ by Corollary~\ref{mbtrg_twin}(a).

(d) Let $G$ be a tree obtained from a $3$-path given by $v_1, v_2, v_3$ by joining exactly two leaves $\ell_i$ and $\ell'_i$ to each $v_i$, where $i\in[3]$. First, we show that $O_{R,1}(G)=\mathcal{N}$. Let $S$ be any distance-$1$ resolving set of $G$; for each $i\in[3]$, $S \cap \{\ell_i, \ell'_i\} \neq \emptyset$ by Observation~\ref{obs_twin}(b) since $\ell_i$ and $\ell'_i$ are twins in $G$. We may assume that $S_0=\{\ell_1, \ell_2, \ell_3\} \subseteq S$ by relabeling the vertices of $G$ if necessary. We note that, for any distinct $i, j\in[3]$, $\code_{S_0,1}(\ell'_i)=\code_{S_0, 1}(\ell'_j)$ and $R_1\{\ell'_i, \ell'_j\}=\{v_i, \ell'_i, v_j, \ell'_j\}$; thus, $S \cap \{v_i, \ell'_i, v_j, \ell'_j\} \neq \emptyset$. So, $|S| \ge5$, and thus $\dim_1(G)\ge5$. Let $X=\{\{v_2, v_3\}\} \cup (\cup_{i=1}^{3} \{\{\ell_i,\ell'_i\}\})$ and let $Z \subseteq V(G)$ with $|Z|=4$ such that $Z \cap\{v_2, v_3\}\neq\emptyset$ and $Z \cap\{\ell_i, \ell'_i\}\neq\emptyset$ for each $i\in[3]$. Since $Z\cup\{v_1\}$ is a minimum distance-$1$ resolving set of $G$, $X$ is a quasi-pairing distance-$1$ resolving set of $G$. In the $B$-game of the MB$1$RG, $B^*$ wins since $M^*$ can occupy at most 4 vertices in the course of the game and $\dim_1(G)=5$. Thus, $O_{R,1}(G)=\mathcal{N}$ by Observation~\ref{pair_resolving}(b).

Second, we show that $O_{R,k}(G)=\mathcal{M}$ for all $k\ge2$. Since $\cup_{i=1}^{3}\{\{\ell_i, \ell'_i\}\}$ is a pairing distance-$k$ resolving set of $G$ for all $k\ge2$, $O_{R,k}(G)=\mathcal{M}$ for all $k\ge2$ by Observation~\ref{pair_resolving}(a).

(e) Let $G$ be a tree obtained from an $\alpha$-path $v_1, v_2, \ldots, v_{\alpha}$, where $\alpha\ge3$, by attaching exactly three leaves $\ell_{\alpha}, \ell'_{\alpha}, \ell''_{\alpha}$ to $v_{\alpha}$ and attaching exactly two leave $\ell_i$ and $\ell'_i$ to each $v_i$, where $i\in[\alpha-1]$.

First, we show that $O_{R,1}(G)=\mathcal{B}$. Let $S$ be any distance-$1$ resolving set of $G$. By Observation~\ref{obs_twin}(b), $|S \cap \{\ell_{\alpha}, \ell'_{\alpha}, \ell''_{\alpha}\}|\ge2$ and $|S \cap \{\ell_i, \ell'_i\}|\ge1$ for each $i\in[\alpha-1]$. By relabeling the vertices of $G$ if necessary, we may assume that $S_0= (\cup_{i=1}^{\alpha-1}\{\ell'_i\}) \cup \{\ell'_{\alpha},\ell''_{\alpha}\}\subseteq S$. Note that, for any distinct $i, j\in[\alpha]$, $\code_{S_0, 1}(\ell_i)=\code_{S_0, 1}(\ell_j)$ and $R_1\{\ell_i, \ell_j\}=\{v_i, \ell_i, v_j, \ell_j\}$; thus, $|S \cap \{v_i, \ell_i, v_j, \ell_j\}|\ge1$. So, $|S| \ge 2\alpha$ and hence $\dim_1(G)\ge 2\alpha$. Since $\dim_1(G)\ge2\alpha\ge\lceil\frac{3\alpha+1}{2}\rceil+1=\lceil\frac{|V(G)|}{2}\rceil+1$ for $\alpha\ge3$, $O_{R,1}(G)=\mathcal{B}$ by Observation~\ref{obs_dimK}(b).

Second, we show that $O_{R,k}(G)=\mathcal{N}$ for all $k\ge2$. Let $k\ge2$. Let $X=\cup_{i=1}^{\alpha}\{\{\ell_i, \ell'_i\}\}$ and let $Z\subseteq V(G)$ with $|Z|=\alpha$ such that $Z \cap \{\ell_i,\ell'_i\}\neq\emptyset$ for each $i\in[\alpha]$. Since $Z \cup \{\ell''_{\alpha}\}$ forms a minimum distance-$k$ resolving set of $G$, $X$ is a quasi-pairing distance-$k$ resolving set of $G$. By Observation~\ref{pair_resolving}(b), $O_{R,k}(G)\in\{\mathcal{M}, \mathcal{N}\}$. Note that, in the $B$-game of the MB$k$RG, $B^*$ has a winning strategy since $B^*$ can occupy two vertices of $\{\ell_{\alpha}, \ell'_{\alpha}, \ell''_{\alpha}\}$ after her second move, and thus preventing $M^*$ from occupying vertices that form a distance-$k$ resolving set of $G$. So, $O_{R,k}(G)=\mathcal{N}$ for all $k\ge2$.

(f) Let $G$ be a tree obtained from an $\alpha$-path $v_1, v_2, \ldots, v_{\alpha}$, where $\alpha\ge4$, by joining exactly two leaves $\ell_i$ and $\ell'_i$ to each $v_i$, where $i\in[\alpha]$.

First, we show that $O_{R,1}(G)=\mathcal{B}$. For each $i\in[\alpha]$ and for any distance-$1$ resolving set $S$ of $G$, $S \cap \{\ell_i, \ell'_i\} \neq \emptyset$ by Observation~\ref{obs_twin}(b). We may assume that $S^*=\cup_{i=1}^{\alpha}\{\ell_i\} \subseteq S$ by relabeling the vertices of $G$ if necessary. We note that, for any distinct $i, j\in[\alpha]$, $\code_{S^*,1}(\ell'_i)=\code_{S^*, 1}(\ell'_j)$ and $R_1\{\ell'_i, \ell'_j\}=\{v_i, \ell'_i, v_j, \ell'_j\}$; thus, $S \cap \{v_i, \ell'_i, v_j, \ell'_j\} \neq \emptyset$. So, $|S| \ge 2\alpha-1$, and thus $\dim_1(G)\ge 2\alpha-1$. Since $\dim_1(G)\ge2\alpha-1\ge\lceil\frac{3\alpha}{2}\rceil+1=\lceil\frac{|V(G)|}{2}\rceil+1$ for $\alpha\ge4$, $O_{R,1}(G)=\mathcal{B}$ by Observation~\ref{obs_dimK}(b).

Second, we show that $O_{R,k}(G)=\mathcal{M}$ for all $k\ge2$. Since $\cup_{i=1}^{\alpha}\{\{\ell_i, \ell'_i\}\}$ is a pairing distance-$2$ resolving set of $G$, $O_{R,2}(G)=\mathcal{M}$ by Observation~\ref{pair_resolving}(a). By Proposition~\ref{monotonicity_2}, $O_{R,k}(G)=\mathcal{M}$ for all $k\ge2$.~\hfill

(g) Let $G$ be the graph in Figure~\ref{fig_outcome_transition2}, where $\alpha\ge2$. We note the following: (i) for any minimum distance-$k$ resolving set $S$ of $G$ and for each $i\in[\alpha]$, $S \cap \{\ell_i, \ell'_i\}\neq\emptyset$ and $S \cap\{s_i, s'_i\}\neq\emptyset$ by Observation~\ref{obs_twin}(b); (ii) if $S_0=\cup_{i=1}^{\alpha}\{\ell_i, s_i\} \subseteq S$, then $\code_{S_0,k}(\ell'_j)=\code_{S_0,k}(s'_j)$ for each $k\ge1$ and for each $j\in[\alpha]$; (iii) for each $i\in[\alpha]$, $R_1\{\ell'_i, s'_i\}=\{\ell'_i, s'_i, x_i\}$; (iv) for each $i\in[\alpha]$, $R_2\{\ell'_i, s'_i\}=\{\ell'_i, s'_i,x_i,y\}$; (v) for each $i\in[\alpha]$ and for each $k\ge3$, $R_k\{\ell'_i, s'_i\}\supseteq\{\ell'_i, s'_i, x_i, y,z\}$. Let $W=\cup_{i=1}^{\alpha}\{\{\ell_i, \ell'_i\}, \{s_i, s'_i\}\}$.

First, we show that $O_{R,1}(G)=\mathcal{B}$. In the $M$-game, $B^*$ can choose exactly one vertex of each pair in $W$ and at least one vertex in $\cup_{i=1}^{\alpha}\{x_i\}$. By relabeling the vertices of $G$ if necessary, we may assume that $B^*$ chose the vertices in $\{x_j\}\cup(\cup_{i=1}^{\alpha}\{\ell'_i, s'_i\})$ after her $(2\alpha+1)$st move. Since all vertices in $\{\ell'_j, s'_j, x_j\}$, for some $j\in[\alpha]$, are occupied by $B^*$, (iii) implies that $M^*$ fails to occupy vertices that form a distance-$1$ resolving set of $G$ in the $M$-game of the MB$1$RG. Thus, $O_{R,1}(G)=\mathcal{B}$.

Second, we show that $O_{R,2}(G)=\mathcal{N}$. Let $U\subseteq V(G)$ with $|U|=2\alpha$ such that $U \cap \{\ell_i, \ell'_i\} \neq\emptyset$ and $U \cap \{s_i, s'_i\} \neq\emptyset$ for each $i\in[\alpha]$. Since $U\cup\{y\}$ is a minimum distance-$2$ resolving set of $G$, $W$ is a quasi-pairing distance-$2$ resolving set of $G$. By Observation~\ref{pair_resolving}(b), $O_{R,2}(G)\in\{\mathcal{M}, \mathcal{N}\}$. In the $B$-game, $B^*$ can select the vertex $y$ and exactly one vertex of each pair in $W$; we may assume that $B^*$ chose the vertices in $\{y\} \cup (\cup_{i=1}^{\alpha}\{\ell'_i, s'_i\})$ after her $(2\alpha+1)$st move. In order for $M^*$ to occupy vertices that form a distance-$2$ resolving set of $G$ in the $B$-game, (iv) implies that $M^*$ must select all vertices in $\cup_{i=1}^{\alpha}\{x_i\}$ in addition to the vertices in $\cup_{i=1}^{\alpha}\{\ell_i, s_i\}$, but this is impossible since $\alpha\ge2$ and $B^*$ can select at least one vertex in $\cup_{i=1}^{\alpha}\{x_i\}$ in her $(2\alpha+2)$nd move. So, $B^*$ wins the $B$-game of the MB$2$RG. Thus, $O_{R,2}(G)=\mathcal{N}$.

Third, we show that $O_{R,k}(G)=\mathcal{M}$ for $k\ge3$. Since $\{\{y,z\}\} \cup W$ is a pairing distance-$k$ resolving set of $G$ for $k\ge3$, $O_{R,k}(G)=\mathcal{M}$ for all $k\ge3$ by Observation~\ref{pair_resolving}(a).~\hfill\end{proof}

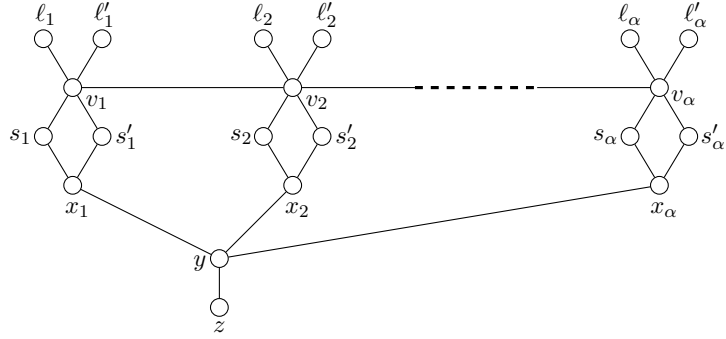
\begin{figure}[ht]
\centering
\begin{tikzpicture}[scale=.65, transform shape]

\node [draw, shape=circle, scale=1] (v1) at  (0,0) {};
\node [draw, shape=circle, scale=1] (v2) at  (4.5,0) {};
\node [draw, shape=circle, scale=1] (v3) at  (12,0) {};

\node [draw, shape=circle, scale=1] (a1) at  (-0.6,1) {};
\node [draw, shape=circle, scale=1] (a2) at  (0.6,1) {};
\node [draw, shape=circle, scale=1] (a3) at  (-0.6,-1) {};
\node [draw, shape=circle, scale=1] (a4) at  (0.6,-1) {};
\node [draw, shape=circle, scale=1] (a5) at  (0,-2) {};

\node [draw, shape=circle, scale=1] (b1) at  (3.9,1) {};
\node [draw, shape=circle, scale=1] (b2) at  (5.1,1) {};
\node [draw, shape=circle, scale=1] (b3) at  (3.9,-1) {};
\node [draw, shape=circle, scale=1] (b4) at  (5.1,-1) {};
\node [draw, shape=circle, scale=1] (b5) at  (4.5,-2) {};

\node [draw, shape=circle, scale=1] (c1) at  (11.4,1) {};
\node [draw, shape=circle, scale=1] (c2) at  (12.6,1) {};
\node [draw, shape=circle, scale=1] (c3) at  (11.4,-1) {};
\node [draw, shape=circle, scale=1] (c4) at  (12.6,-1) {};
\node [draw, shape=circle, scale=1] (c5) at  (12,-2) {};

\node [draw, shape=circle, scale=1] (y) at  (3,-3.5) {};
\node [draw, shape=circle, scale=1] (z) at  (3,-4.5) {};

\draw(v1)--(v2)--(7,0);\draw(9.5,0)--(v3);\draw[very thick, dashed](7,0)--(9.5,0);
\draw(a1)--(v1)--(a2);\draw(v1)--(a3)--(a5)--(a4)--(v1);
\draw(b1)--(v2)--(b2);\draw(v2)--(b3)--(b5)--(b4)--(v2);
\draw(c1)--(v3)--(c2);\draw(v3)--(c3)--(c5)--(c4)--(v3);
\draw(a5)--(y)--(b5);\draw(c5)--(y)--(z);

\node [scale=1.3] at (0.5,-0.3) {$v_1$};
\node [scale=1.3] at (5,-0.3) {$v_2$};
\node [scale=1.3] at (12.5,-0.2) {$v_{\alpha}$};

\node [scale=1.3] at (-0.55,1.5) {$\ell_1$};
\node [scale=1.3] at (0.65,1.5) {$\ell'_1$};
\node [scale=1.3] at (-1.07,-1) {$s_1$};
\node [scale=1.3] at (1.1,-1) {$s'_1$};
\node [scale=1.3] at (0.1,-2.5) {$x_1$};

\node [scale=1.3] at (3.9,1.5) {$\ell_2$};
\node [scale=1.3] at (5.2,1.5) {$\ell'_2$};
\node [scale=1.3] at (3.45,-1) {$s_2$};
\node [scale=1.3] at (5.6,-1) {$s'_2$};
\node [scale=1.3] at (4.6,-2.5) {$x_2$};

\node [scale=1.3] at (11.4,1.5) {$\ell_{\alpha}$};
\node [scale=1.3] at (12.75,1.5) {$\ell'_{\alpha}$};
\node [scale=1.3] at (10.9,-1) {$s_{\alpha}$};
\node [scale=1.3] at (13.1,-1) {$s'_{\alpha}$};
\node [scale=1.3] at (12.1,-2.5) {$x_{\alpha}$};

\node [scale=1.3] at (2.6,-3.6) {$y$};
\node [scale=1.3] at (3,-4.9) {$z$};

\end{tikzpicture}
\caption{\small{Graphs $G$ such that $O_{R,1}(G)=\mathcal{B}$, $O_{R,2}(G)=\mathcal{N}$ and $O_{R,k}(G)=\mathcal{M}$ for $k\ge3$, where $\alpha\ge2$.}}\label{fig_outcome_transition2}
\end{figure}

%%%%%%%%%%%%%%%%%%%%%%%%%%%%%%%
%%%%%%%%%%%%%%%%%%%%%%%%%%%%%%%

\section{Time to win or lose the MB$k$RG, and some graph classes}\label{sec_examples}

In this section, we consider the outcome of the MB$k$RG as well as the minimum number of steps needed to reach the outcome of the MB$k$RG on some graph classes. Taking into consideration of Proposition~\ref{monotonicity_2} and Corollary~\ref{mbrg_mbtrg}, we make the following observation.

\begin{observation}\label{obs_time}
Let $k\in\mathbb{Z}^+$ and $G$ be a connected graph of order $n\ge2$.
\begin{itemize}
\item[(a)] If $O_{R,k}(G)=\mathcal{M}$, then $\dim_k(G) \le M_{R,k}(G) \le M'_{R,k}(G) \le \lfloor\frac{n}{2}\rfloor$, $M_{R,k+1}(G)\le M_{R,k}(G)$ and $M'_{R,k+1}(G)\le M'_{R,k}(G)$.
\item[(b)] If $O_{R,k}(G)=\mathcal{B}$, then $B'_{R,k}(G)\le B_{R,k}(G) \le \lfloor\frac{n}{2}\rfloor$; moreover, for $k>1$, $B_{R,k}(G) \ge B_{R,k-1}(G)$ and $B'_{R,k}(G)\ge B'_{R, k-1}(G)$.
\item[(c)] If $\diam(G)\in\{1,2\}$ or $k\ge \diam(G)-1$, then $M_{R,k}(G)=M_R(G)$ and $M'_{R,k}(G)=M'_R(G)$ if $O_R(G)=\mathcal{M}$, and $B_{R,k}(G)=B_R(G)$ and $B'_{R,k}(G)=B'_R(G)$ if $O_R(G)=\mathcal{B}$.
\item[(d)] If $X$ is a pairing distance-$k$ resolving set of $G$ with $|\bigcup X|=2\dim_k(G)$, then $M_{R,k}(G)=M'_{R,k}(G)=\dim_k(G)$.
\item[(e)] If $X$ is a quasi-pairing distance-$k$ resolving set of $G$ with $|\bigcup X|=2(\dim_k(G)-1)$ and $O_{R,k}(G)=\mathcal{N}$, then $N_{R,k}(G)=\dim_k(G)$.
\end{itemize}
\end{observation}

Next, we examine the MB$k$RG on some graph classes. We first consider the Petersen graph and complete multipartite graphs.

\begin{example}\label{example_diam2}
Let $k\in\mathbb{Z}^+$.

(a) Let $\mathcal{P}$ denote the Petersen graph. It was shown in~\cite{mbrg} that $O_R(\mathcal{P})=\mathcal{M}$ and $M_R(\mathcal{P})=3=M'_R(\mathcal{P})$. Since $\diam(\mathcal{P})=2$, Corollary~\ref{mbrg_mbtrg}(a) and Observation~\ref{obs_time}(c) imply that $O_{R,k}(\mathcal{P})=\mathcal{M}$ and $M_{R,k}(\mathcal{P})=3=M'_{R,k}(\mathcal{P})$ for all $k$.

(b) For $m\ge2$, let $G=K_{a_1, a_2, \ldots, a_m}$ be a complete multi-partite graph of order $\sum_{i=1}^{m}a_i$, and let $s$ be the number of partite sets of $G$ consisting of exactly one element. Since $\diam(G)\le2$, Corollary~\ref{mbrg_mbtrg}(a) implies $O_{R,k}(G)=O_R(G)$ for all $k$, and $O_R(G)$ was determined in~\cite{mbrg}. So,
\begin{equation*}
O_{R,k}(G)=\left\{
\begin{array}{ll}
\mathcal{B} & \mbox{if }s\ge4, \mbox{ or }a_i\ge 4 \mbox{ for some }i\in[m],\\
 & \mbox{ or }s=a_i=3 \mbox{ for some }i\in [m],\\
 & \mbox{ or }a_i=a_j=3 \mbox{ for distinct }i,j\in[m],\\
\mathcal{N} & \mbox{if }s=3 \mbox{ and }a_i \le 2 \mbox{ for each }i\in[m],\\
 & \mbox{ or } s\le 2, a_i\le3 \mbox{ for each }i\in[m] \mbox{ and } a_j\!=\!3 \mbox{ for exactly one }j\in[m],\\
\mathcal{M} & \mbox{if } s \le 2 \mbox{ and }a_i \le 2 \mbox{ for each } i\in[m].
\end{array}\right.
\end{equation*}
Moreover, we have the following: (i) if $O_{R,k}(G)=\mathcal{M}$, then $M_{R,k}(G)=M'_{R,k}(G)=\dim(G)$ by Theorem~4.12 of~\cite{mbrg} and Observation~\ref{obs_time}(c); (ii) if $O_{R,k}(G)=\mathcal{B}$, then $B_{R,k}(G)=B'_{R,k}(G)=2$ by Proposition~3.2 of~\cite{mbrg} and Observation~\ref{obs_time}(c); (iii) if $O_{R,k}(G)=\mathcal{N}$, then $N_{R,k}(G)=\dim(G)$ and $N'_{R,k}(G)=2$. To see (iii), we note that if $O_{R,k}(G)=\mathcal{N}$, then $G$ has exactly one twin equivalence class of cardinality $3$, say $Q$, and $G$ admits a quasi-pairing distance-$k$ resolving set, say $X$, with $|\bigcup X|=2(\dim_k(G)-1)$. In the $M$-game, $N_{R,k}(G)=\dim(G)$ by Observations~\ref{obs_dimK_diameter}(a) and~\ref{obs_time}(e). In the $B$-game, $N'_{R,k}(G)=2$ by Observation~\ref{obs_twin}(b) and the fact that $B^*$ can occupy $2$ vertices of $Q$ after her second move.
\end{example}

Next, we consider cycles. It was obtained in~\cite{mbrg} that $O_R(C_n)=\mathcal{M}$ for $n\ge4$. We recall some terminology. Following~\cite{wheel1}, let $M$ be a set of at least two vertices of $C_n$, let $u_i$ and $u_j$ be distinct vertices of $M$, and let $P$ and $P'$ denote the two distinct $u_i-u_j$ paths determined by $C_n$. If either $P$ or $P'$, say $P$, contains only two vertices of $M$ (namely, $u_i$ and $u_j$), then we refer to $u_i$ and $u_j$ as \emph{neighboring vertices} of $M$ and the set of vertices of $V(P)-\{u_i,u_j\}$ as the \emph{gap} of $M$ (determined by $u_i$ and $u_j$). The two gaps of $M$ determined by a vertex of $M$ and its two neighboring vertices of $M$ are called \emph{neighboring gaps}. Note that, $M$ has $r$ gaps if $|M|=r$, where some of the gaps may be empty. The following lemma and its proof are adapted from~\cite{wheel1}.  

\begin{lemma}\label{lem_cycle} 
Let $S \subseteq V(C_n)$, where $n\ge5$. Suppose $S$ satisfies the following two conditions: (1) every gap of $S$ contains at most $3$ vertices, and at most one gap of $S$ contains $3$ vertices; (2) if a gap of $S$ contains at least $2$ vertices, then its neighboring gaps contain at most $1$ vertex. Then $S$ is a distance-$1$ resolving set of $C_n$.
\end{lemma}

\begin{proof}
Let $S \subseteq V(C_n)$ satisfy the conditions of the present lemma, where $n\ge5$. Let $C_n$ be given by $u_0, u_1, \ldots, u_{n-1}, u_0$ and let $v\in V(C_n)-S$. Let ${\bf{2}}_{|S|}$ denote the $|S|$-vector with $2$ on each entry of $\code_{S,1}(\cdot)$, and all subscripts in this proof are taken modulo $n$.

First, suppose $v=u_{i}$ belongs to a gap of size $1$ of $S$. Then $W_0=\{u_{i-1}, u_{i+1}\} \subseteq S$ and $\code_{W_0,1}(u_{i})=(1,1) \neq \code_{W_0,1}(u_j)$ for each $u_j\in V(C_n)-(S\cup \{u_i\})$ since $u_j$ cannot be adjacent to both $u_{i-1}$ and $u_{i+1}$ in $C_n$.

Second, suppose $v=u_i$ belongs to a gap of size $2$ of $S$ such that $\{u_i, u_{i+1}\}\cap S=\emptyset$ or $\{u_{i-1}, u_i\}\cap S=\emptyset$, say the former. Let $W_1=\{u_{i-1}, u_{i+2}\} \subseteq S$. If $n=5$, then $\{\code_{W_1, 1}(u_j): u_j \in V(C_5) -W_1\} =\{(1,1),(1,2),(2,1)\}$. So, suppose $n\ge6$. By the condition (2) of the present lemma, $\{u_{i-2},u_{i-3}\} \cap S \neq\emptyset$. If $u_{i-2} \in S$, then $\code_{W_1,1}(u_{i})=(1,2) \neq \code_{W_1,1}(u_j)$ for each $u_j\in V(C_n)-(S\cup \{u_i\})$ since $u_j$ cannot be adjacent to $u_{i-1}$ in $C_n$. If $u_{i-2} \not\in S$ and $u_{i-3} \in S$, then we have the following: (i) $\code_{W_1,1}(u_i)=\code_{W_1,1}(u_{i-2})=(1,2) \neq \code_{W_1,1}(u_j)$ for each $u_j\in V(C_n)-(S\cup \{u_i, u_{i-2}\})$ since $u_j$ is not adjacent to $u_{i-1}$ in $C_n$; (ii) $\code_{S,1}(u_i)\neq \code_{S,1}(u_{i-2})$ since $d_1(u_i, u_{i-3})=2>1=d_1(u_{i-2}, u_{i-3})$.

Third, suppose $v=u_i$ belongs to a gap of size $3$ of $S$. If $\{u_{i-1},u_i, u_{i+1}\} \cap S=\emptyset$, then $\code_{S,1}(u_i)={\bf{2}}_{|S|} \neq \code_{S,1}(u_j)$ for each $u_j\in V(C_n)-(S \cup \{u_i\})$ since $u_j$ is adjacent to at least one vertex of $S$. Now, suppose $\{u_i, u_{i+1}, u_{i+2}\} \cap S=\emptyset$ or $\{u_{i-2}, u_{i-1}, u_i\}\cap S=\emptyset$, say the former; then $W=\{u_{i-1}, u_{i+3}\} \subseteq S$. If $n=5$, then $\{\code_{W, 1}(u_j): u_j \in V(C_5) -W\} =\{(1,2),(2,1), (2,2)\}$. If $n=6$, then $\{\code_{W, 1}(u_j): u_j \in V(C_6) -W\} =\{(1,1),(1,2),(2,1), (2,2)\}$. So, suppose $n\ge7$. Note that $\{u_{i-2}, u_{i-3}\}\cap S \neq\emptyset$ by the condition (2) of the present lemma. If $u_{i-2}\in S$, then $\code_{W,1}(u_i)=(1,2) \neq \code_{W, 1}(u_j)$ for each $u_j\in V(C_n)-(S\cup \{u_i\})$ since $u_j$ is not adjacent to $u_{i-1}$ in $C_n$. If $u_{i-2}\not\in S$ and $u_{i-3} \in S$, then $\code_{W, 1}(u_i)=\code_{W,1}(u_{i-2})=(1,2) \neq \code_{W,1}(u_j)$ for each $u_j\in V(C_n)-(S\cup \{u_i, u_{i-2}\})$ and $d_1(u_i,u_{i-3})=2>1=d_1(u_{i-2}, u_{i-3})$; thus, $\code_{S,1}(u_i)\neq \code_{S,1}(u_j)$ for each $u_j\in V(C_n)-(S\cup \{u_i\})$.~\hfill
\end{proof}

As a generalization of Lemma~\ref{lem_cycle}, we state the following result without providing a detailed proof, where its proof for $k=1$ is given in Lemma~\ref{lem_cycle}. The proof for the converse of Lemma~\ref{lem_cycle2} is provided in~\cite{distKdim}.

\begin{lemma}\label{lem_cycle2}
Let $k\in\mathbb{Z}^+$ and $S \subseteq V(C_n)$, where $n\ge2k+3$. Suppose $S$ satisfies the following two conditions: (1) every gap of $S$ contains at most $(2k+1)$ vertices, and at most one gap of $S$ contains $(2k+1)$ vertices; (2) if a gap of $S$ contains at least $(k+1)$ vertices, then its neighboring gaps contain at most $k$ vertices. Then $S$ is a distance-$k$ resolving set of $C_n$.
\end{lemma}

\begin{proposition}\label{mbtrg1_cycle}
Let $k\in\mathbb{Z}^+$ and $n\ge 3$. Then
\begin{equation*}
O_{R,k}(C_n)=\left\{
\begin{array}{ll}
\mathcal{N} & \mbox{if $n=3$ and~}k\ge1,\\
\mathcal{M} & \mbox{if $n\ge4$ is even and~}k\ge1,\\
\mathcal{M} & \mbox{if $n\ge5$ is odd and~}k\ge2.
\end{array}\right.
\end{equation*}
\end{proposition}

\begin{proof}
Let $k\in\mathbb{Z}^+$. Let $C_n$ be given by $u_1, u_2, \ldots, u_{n}, u_1$, where $n\ge 3$. Note that $O_{R,k}(C_3)=\mathcal{N}$ by Example~\ref{example_diam2}(b). If $n=4$, then $\{\{u_1, u_3\}, \{u_2,u_4\}\}$ is a pairing distance-$k$ resolving set of $C_4$; thus, $O_{R,k}(C_4)=\mathcal{M}$ by Observation~\ref{pair_resolving}(a). So, suppose $n\ge5$. 

If $n=2x$ ($x\ge3$), let $X=\cup_{i=1}^{x}\{\{u_{2i-1}, u_{2i}\}\}$ and $S$ be the set of vertices that are selected by $M^*$ over the course of the MB$1$RG such that $M^*$ selects exactly one vertex in $\{u_{2i-1}, u_{2i}\}$ for each $i\in[x]$. Then we have the following: (i) every gap of $S$ contains at most $2$ vertices; (ii) if a gap of $S$ contains $2$ vertices, then its neighboring gaps contain at most $1$ vertex. By Lemma~\ref{lem_cycle}, $S$ is a distance-$1$ resolving set of $C_n$. Since $X$ is a pairing distance-$1$ resolving set of $C_n$, $O_{R,1}(C_n)=\mathcal{M}$ by Observation~\ref{pair_resolving}(a). By Proposition~\ref{monotonicity_2}, $O_{R,k}(C_n)=\mathcal{M}$ for even $n\ge6$ and for $k\ge1$.

If $n=2x+1$ ($x\ge2$), let $Y=\cup_{i=1}^{x}\{\{u_{2i-1}, u_{2i}\}\}$. In the $B$-game of the MB$k$RG, suppose $B^*$ selects $u_n=u_{2x+1}$ (by relabeling the vertices of $C_n$ if necessary) after her first move and let $S'$ be the set of vertices that are selected by $M^*$ such that $M^*$ selects exactly one vertex in $\{u_{2i-1}, u_{2i}\}$ for each $i\in[x]$. Then we have the following: (i) at most one gap of $S'$ contains $3$ vertices (when $B^*$ is able to select all vertices in $\{u_1,u_{n-1},u_n\}$ over the course of the MB$k$RG) and all other gaps of $S'$ contain at most $2$ vertices; (ii) if a gap of $S'$ contains $3$ vertices, then its neighboring gaps contain at most $1$ vertex. By Lemma~\ref{lem_cycle2}, $S'$ is a distance-$2$ resolving set of $C_n$. Since $Y$ is a pairing distance-$2$ resolving set of $C_n$, $O_{R,2}(C_n)=\mathcal{M}$ by Observation~\ref{pair_resolving}(a). By Proposition~\ref{monotonicity_2}, $O_{R,k}(C_n)=\mathcal{M}$ for odd $n\ge5$ and for $k\ge2$.~\hfill
\end{proof}

\begin{remark}\label{rem_cycle_57}
Let $C_n$ be given by $u_1, u_2, \ldots, u_n, u_1$. We consider the $B$-game of the MB$1$RG.

(1) We show that $O_{R,1}(C_5)=\mathcal{M}$. Without loss of generality, suppose $B^*$ selects $u_5$ on her first move. Then $M^*$ selects $u_4$ on his first move and exactly one vertex in $\{u_2, u_3\}$ on his second move. Since the vertices selected by $M^*$ form a distance-$1$ resolving set of $C_5$, $O_{R,1}(C_5)=\mathcal{M}$. 

(2) We show that $O_{R,1}(C_7)=\mathcal{M}$. Let $S$ be the set of vertices that are selected by $M^*$ over the course of the game. Without loss of generality, suppose $B^*$ selects $u_7$ on her first move. Then $M^*$ selects $u_6$ on his first move. If $B^*$ selects a vertex in $\{u_1, u_2\}$ on her second move, then $M^*$ selects $u_3$ on his second move and exactly one vertex of $\{u_4, u_5\}$ on his third move. If $B^*$ selects $u_3$ on her second move, then $M^*$ selects $u_4$ on his second move and exactly one vertex of $\{u_1, u_2\}$ on his third move. If $B^*$ selects $u_4$ on her second move, then $M^*$ selects $u_3$ on his second move and exactly one vertex of $\{u_1, u_5\}$ on his third move. If $B^*$ selects $u_5$ on her second move, then $M^*$ selects $u_4$ on his second move and exactly one vertex of $\{u_1,u_2, u_3\}$ on his third move. In each case, $S$ satisfies the conditions of Lemma~\ref{lem_cycle}; thus, $S$ is a distance-$1$ resolving set of $C_7$ and $O_{R,1}(C_7)=\mathcal{M}$. \end{remark}

\begin{conjecture}
In addition to $C_5$ and $C_7$, we find that $O_{R,1}(C_9)=\mathcal{M}$ through explicit computation. We further conjecture that $O_{R,1}(C_n)=\mathcal{M}$ for all odd $n\ge5$, but an argument for the general (odd) $n$ eludes us.
\end{conjecture}

Next, we consider wheel graphs. The \emph{join} of two graphs $G$ and $H$, denoted by $G+H$, is the graph obtained from the disjoint union of $G$ and $H$ by adding additional edges between each vertex of $G$ and each vertex of $H$. Since $\diam(G+H)\le2$, $O_{R,k}(G+H)=O_R(G+H)$, for all $k\in\mathbb{Z}^+$, by Corollary~\ref{mbrg_mbtrg}(a). Let $d_G(w_i, w_j)$ denote the distance between the vertices $w_i$ and $w_j$ in a graph $G$, and let $d_{G,k}(w_i, w_j)$ denote $d_k(w_i, w_j)$ in $G$.

For the wheel graph $C_n+K_1$, let $V(K_1)=\{v\}$ and let $C_n$ be given by $u_1, u_2, \ldots, u_n, u_1$ such that $v$ is adjacent to each $u_i$, where $i\in[n]$, in $C_n+K_1$. Let $k\in\mathbb{Z}^+$. If $n=3$, then $O_{R,k}(C_3+K_1)=O_{R,k}(K_4)=\mathcal{B}$ by Proposition~\ref{mbtrg_twin}(a).
If $n=4$, then $\{\{u_1, u_3\}, \{u_2, u_4\}\}$ is a pairing distance-$k$ resolving set of $C_4+K_1$. If $n=5$, then  $\{\{u_1, u_2\}, \{u_3, u_4\}, \{u_5, v\}\}$ is a pairing distance-$k$ resolving set of $C_5+K_1$. So, $O_{R,k}(C_4+K_1)=O_{R,k}(C_5+K_1)=\mathcal{M}$, for all $k\ge1$, by Observation~\ref{pair_resolving}(a). For $n \ge6$, let $S$ be the set of vertices lying on $C_n$ that are selected by $M^*$ over the course of MB$k$RG; notice that $\code_{S,k}(v)=(1,1,\ldots,1)\neq \code_{S,k}(u_i)$ for each $i\in[n]$. Since $d_{C_n+K_1}(u_i, u_j)=d_{C_n,1}(u_i, u_j)$ for $i, j\in [n]$, the same method of proof for Proposition~\ref{mbtrg1_cycle} and Remark~\ref{rem_cycle_57} provides the following result. 

\begin{corollary}
Let $k\in\mathbb{Z}^+$ and $n\ge3$. Then (i) $O_{R,k}(C_3+K_1)=O_{R}(C_3+K_1)=\mathcal{B}$; (ii) if $n\in \{4,5,6,7\}$ or $n\ge8$ is even, then $O_{R,k}(C_n+K_1)=O_{R}(C_n+K_1)=\mathcal{M}$; (iii) if $n\ge9$ is odd, then $O_{R,k}(C_n+K_1)=O_{R}(C_n+K_1) \in \{\mathcal{M}, \mathcal{N}\}$ (because $M^*$ has a winning strategy in the $M$-game of the MB$k$RG).
\end{corollary}

Next, we consider trees with some restrictions. We recall some terminology and notations. Fix a tree $T$. A vertex $\ell$ of degree one is called a \emph{terminal vertex} of a major vertex $v$ if $d(\ell, v)<d (\ell, w)$ for every other major vertex $w$ in $T$. The \emph{terminal degree}, $ter_T(v)$, of a major vertex $v$ is the number of terminal vertices of $v$ in $T$, and an \emph{exterior major vertex} is a major vertex with positive terminal degree. Let $M(T)$ be the set of exterior major vertices of $T$. For each $i\in[3]$, let $M_i(T)=\{w\in M(T): ter_T(w)=i\}$, and let $M_4(T)=\{w \in M(T):ter_T(w) \ge 4\}$; note that $M(T)=\cup_{j=1}^{4}M_j(T)$.

\begin{theorem}
Let $k\in\mathbb{Z}^+$ and $T$ be a tree that is not a path. Further, suppose that $T$ contains neither a degree-two vertex nor a major vertex with terminal degree zero. Then
\begin{equation*}
O_{R,k}(T)=\left\{
\begin{array}{ll}
\mathcal{B} & \mbox{if } |M_4(T)|\ge 1 \mbox{ or } |M_3(T)|\ge 2 \mbox{ for all }k\ge1,\\
{} & \mbox{ or } M_4(T)=\emptyset, |M_3(T)|=1, |M_2(T)|\ge2 \mbox{ and } k=1,\\
{} & \mbox{ or } M_4(T)=M_3(T)=\emptyset, |M_2(T)|\ge 4 \mbox{ and } k=1,\\
\mathcal{N} & \mbox{if } M_4(T)=\emptyset, |M_3(T)|=1, |M_2(T)|\in\{0,1\}\mbox{ and } k=1,\\
{}  & \mbox{ or } M_4(T)=\emptyset \mbox{ and } |M_3(T)|=1 \mbox{ for all }k\ge2,\\
{} & \mbox{ or } M_4(T)=M_3(T)=\emptyset, |M_2(T)|= 3 \mbox{ and } k=1,\\
\mathcal{M} & \mbox{if } M_4(T)=M_3(T)=\emptyset, |M_2(T)|= 2 \mbox{ and } k=1,\\
{} & \mbox{ or } M_4(T)=M_3(T)=\emptyset \mbox{ and } M_2(T)\neq\emptyset \mbox{ for all }k\ge2.
\end{array}\right.
\end{equation*}
\end{theorem}

\begin{proof}
Let $k\in\mathbb{Z}^+$, and let $T$ be a tree as described in the statement of the present theorem. Then each vertex of $T$ is either a leaf or an exterior major vertex. Let $M_1(T)=\{u_1, u_2, \ldots, u_x\}$ and $M'(T)=\cup_{i=2}^{4}M_i(T)=\{v_1, v_2, \ldots, v_z\}$, where $x\ge0$ and $z\ge1$. If $x\ge1$, then, for each $j\in[x]$, let $m_j$ be the terminal vertex of $u_j$; notice that $m_ju_j\in E(T)$. For each $i\in[z]$, let $ter_T(v_i)=\sigma_i\ge2$ and let $\{\ell_{i,1}, \ldots, \ell_{i, \sigma_i}\}$ be the set of terminal vertices of $v_i$ in $T$.

First, suppose $|M_4(T)|\ge1$; then Corollary~\ref{mbtrg_twin}(a) implies that $O_{R,k}(T)=\mathcal{B}$ for all $k\ge1$. Second, suppose $|M_3(T)|\ge2$; then Corollary~\ref{mbtrg_twin}(b) implies that $O_{R,k}(T)=\mathcal{B}$ for all $k\ge1$.

Third, suppose that $M_4(T)=\emptyset$ and $|M_3(T)|=1$; then $z=1+|M_2(T)|$. We may assume that $ter_T(v_z)=3$ by relabeling the vertices of $T$ if necessary. If $z=1$ (i.e., $M_2(T)=\emptyset$), then $T=K_{1,3}$ and $O_{R,k}(T)=\mathcal{N}$ for all $k\ge1$ by Example~\ref{example_diam2}(b).

Now, suppose that $z\ge2$ (i.e., $M_2(T)\neq\emptyset$); then $ter_T(v_i)=2$ for each $i\in[z-1]$. For any minimum distance-$k$ resolving set $W$ of $T$, Observation~\ref{obs_twin}(b) implies that $|W\cap\{\ell_{z,1}, \ell_{z,2}, \ell_{z,3}\}|\ge2$ and $|W \cap \{\ell_{i,1}, \ell_{i,2}\}|\ge1$ for each $i\in[z-1]$. By relabeling the vertices of $T$ if necessary, let $W_0=(\cup_{i=1}^{z-1}\{\ell_{i,1}\}) \cup\{\ell_{z,1}, \ell_{z,3}\} \subseteq W$. Then, for any distinct $i,j\in[z]$ and for any distinct $\alpha, \beta\in[x]$, we have the following: (i) $\code_{W_0,1}(\ell_{i,2})=\code_{W_0,1}(\ell_{j,2})=\code_{W_0,1}(m_{\alpha})=\code_{W_0,1}(m_{\beta})$, $R_1\{\ell_{i,2}, \ell_{j,2}\}=\{\ell_{i,2}, v_i, \ell_{j,2}, v_j\}$, $R_1\{m_{\alpha}, m_{\beta}\}=\{m_{\alpha}, u_{\alpha}, m_{\beta}, u_{\beta}\}$ (if $x\ge2$), and $R_1\{\ell_{i,2}, m_{\alpha}\}=\{\ell_{i,2}, v_i, m_{\alpha}, u_{\alpha}\}$ (if $x\ge1$); (ii) for $k\ge2$, $\code_{W_0,k}(\ell_{i,2})\neq\code_{W_0,k}(\ell_{j,2})$. If $k=1$ and $z=2$, then the first player has a winning strategy: (i) in the $M$-game, $M^*$ can select $\ell_{2,3}$ and then exactly one vertex of each pair in $\{\{\ell_{1,1}, \ell_{1,2}\}, \{\ell_{2,1},\ell_{2,2}\}, \{v_1, v_2\}\} \cup (\cup_{i=1}^{x}\{\{u_i, m_i\}\})$ thereafter, and thus occupying a distance-$1$ resolving set of $T$ in the course of the MB$1$RG; (ii) in the $B$-game, $B^*$ can occupy two vertices of $\{\ell_{2,1}, \ell_{2,2}, \ell_{2,3}\}$ after her second move, and thus preventing $M^*$ from occupying vertices that form any distance-$1$ resolving set of $T$ in the course of the MB$1$RG. If $k=1$ and $z\ge3$, then $\dim_1(G)\ge2z+x\ge\lceil\frac{1+3z+2x}{2}\rceil+1=\lceil\frac{|V(T)|}{2}\rceil+1$; thus, $O_{R,1}(T)=\mathcal{B}$ by Observation~\ref{obs_dimK}(b). If $k\ge2$, then $A=(\cup_{i=1}^{z}\{\{\ell_{i,1}, \ell_{i,2}\}\}) \cup (\cup_{j=1}^{x}\{\{u_j, m_j\}\})$ is a quasi-pairing distance-$k$ resolving set of $T$, and the first player has a winning strategy in the MB$k$RG: (i) in the $M$-game, $M^*$ can occupy $\ell_{z,3}$ after his first move and exactly one vertex of each pair in $A$ thereafter; (ii) in the $B$-game, $B^*$ can occupy two vertices of $\{\ell_{z,1}, \ell_{z,2}, \ell_{z,3}\}$ after her second move.

Fourth, suppose that $M_4(T)=\emptyset=M_3(T)$. Since $T$ is not a path, $M_2(T)\neq\emptyset$; notice that $|M_2(T)|=z\ge2$ in this case. Let $C=(\cup_{i=1}^{z}\{\{\ell_{i,1}, \ell_{i,2}\}\}) \cup (\cup_{j=1}^{x}\{\{u_j, m_j\}\})$. If $k\ge2$, then $C$ is a pairing distance-$k$ resolving set of $T$; thus, $O_{R,k}(T)=\mathcal{M}$ by  Observation~\ref{pair_resolving}(a). So, suppose $k=1$. If $z=2$, then $C\cup \{\{v_1, v_2\}\}$ is a pairing distance-$1$ resolving set of $T$; thus, $O_{R,1}(T)=\mathcal{M}$ by  Observation~\ref{pair_resolving}(a). If $z=3$, then $O_{R,1}(T)=\mathcal{N}$: (i) in the $B$-game of the MB$1$RG, $B^*$ has a winning strategy since $\dim_1(T) \ge x+5=\lceil\frac{2x+9}{2}\rceil=\lceil\frac{|V(T)|}{2}\rceil>\frac{|V(T)|}{2}$ (via an argument similar to the third case of the present proof); (ii) in the $M$-game of the MB$1$RG, $M^*$ can select $v_1$ and then exactly one vertex of each pair in $C \cup\{\{v_2, v_3\}\}$ thereafter, and thus occupying vertices that form a distance-$1$ resolving set of $T$. If $z\ge4$, then $\dim_1(T)\ge 2z+x-1\ge \lceil\frac{3z+2x}{2}\rceil+1=\lceil\frac{|V(T)|}{2}\rceil+1$ using a similar argument shown in the third case of the present proof; thus, $O_{R,1}(T)=\mathcal{B}$ by Observation~\ref{obs_dimK}(b).~\hfill
\end{proof}

We conclude this section with a couple of questions.

\begin{question}
Among connected graphs $G$ of a fixed diameter $\lambda$, how many of the $\lambda+1 \choose 2$ ternary-valued, monotone functions on $[\lambda-1]$ are realized as $O_{R,k}(G)$? Which ones?
(We speculate that answers to the questions are more accessible for certain classes of graphs such as trees.)
\end{question}

\begin{question}
For a connected graph $G$ and each $k\in[\diam(G)-1]$, is there an algorithm to determine the time ($M_{R,k}$, $M'_{R,k}$, $N_{R,k}$, $N'_{R,k}$, $B_{R,k}$, and $B'_{R,k}$) for the Maker or the Breaker to win MB$k$RG?
\end{question}

\textbf{Acknowledgements.} The authors thank the anonymous referees for their helpful comments.

%%%%%%%%%%%%%%%%%%%%%%%%%%%%%%%
%%%%%%%%%%%%%%%%%%%%%%%%%%%%%%%


\begin{thebibliography}{0}

\bibitem{beardon} A.F. Beardon and J.A. Rodr\'{i}guez-Vel\'{a}zquez, On the $k$-metric dimension of metric spaces. \textit{Ars Math. Contemp.} \textbf{16} (2019) 25-38.

\bibitem{beck} J. Beck, \textit{Combinatorial Games. Tic-Tac-Toe Theory.} Cambridge University Press, Cambridge, 2008.

\bibitem{network} Z. Beerliova, F. Eberhard, T. Erlebach, A.Hall, M. Hoffmann, M. Mihal\'{a}k and L.S. Lam, Network discovery and verification. \textit{IEEE J. Sel. Areas Commun.} \textbf{24} (2006) 2168-2181.

\bibitem{wheel1} P.S. Buczkowski, G. Chartrand, C. Poisson and P. Zhang, On $k$-dimensional graphs and their bases. \textit{Period. Math. Hungar.} \textbf{46} (2003) 9-15.

\bibitem{mbdg} E. Duch\^{e}ne, V. Gledel, A. Parreau and G. Renault, Maker-Breaker domination game. \textit{Discrete Math.} \textbf{343}(9) (2020) \#111955.

\bibitem{erdos} P. Erd\"{o}s and J.L. Selfridge, On a combinatorial game. \textit{J. Combin. Theory Ser. A} \textbf{14} (1973) 298-301.

\bibitem{moreno_thesis} A. Estrada-Moreno, On the $(k,t)$-metric dimension of a graph. Ph.D. thesis, Universitat Rovira i Virgili, 2016.

\bibitem{juan1} A. Estrada-Moreno, I.G. Yero and J.A. Rodr\'{i}guez-Vel\'{a}zquez, On the $(k,t)$-metric dimension of graphs. \textit{Comput. J.} \textbf{64} (2021) 707-720.

\bibitem{juan2} H. Fernau and J.A. Rodr\'{i}guez-Vel\'{a}zquez, On the (adjacency) metric dimension of corona and strong product graphs and their local variants: combinatorial and computational results. \textit{Discrete Appl. Math.} \textbf{236} (2018) 183-202.

\bibitem{distKdim} R.M. Frongillo, J. Geneson, M.E. Lladser, R.C. Tillquist and E. Yi, Truncated metric dimension for finite graphs. \textit{Discrete Appl. Math.} \textbf{320} (2022) 150-169.

\bibitem{hex} M. Gardner, \textit{The Scientific American Book of Mathematical Puzzles and Diversions}. Simon \& Schuster, New York, 1959, pp. 73-83.

\bibitem{NP} M.R. Garey and D.S. Johnson, \textit{Computers and Intractability: A Guide to the Theory of NP-Completeness}. Freeman, New York, 1979.

\bibitem{JE} J. Geneson and E. Yi, The distance-$k$ dimension of graphs. arXiv:2106.08303v2 (2021) https://arxiv.org/abs/2106.08303

\bibitem{harary} F. Harary and R.A. Melter, On the metric dimension of a graph. \textit{Ars Combin.} \textbf{2} (1976) 191-195.

\bibitem{hefetz} D. Hefetz, M. Krivelevich, M. Stojakovi\'{c} and T. Szab\'{o}, \textit{Positional Games.} Birkh\"{a}user/Springer, Basel, 2014.

\bibitem{Hernando}  C. Hernando, M. Mora, I.M. Pelayo, C. Seara and D.R. Wood, Extremal graph theory for metric dimension and diameter. \textit{Electron. J. Combin.} \textbf{17} (2010) \#R30.

\bibitem{adim} M. Jannesari and B. Omoomi, The metric dimension of the lexicographic product of graphs. \textit{Discrete Math.} \textbf{312} (2012) 3349-3356.

\bibitem{mbrg} C.X. Kang, S. Klav\v{z}ar, I.G. Yero and E. Yi, Maker-Breaker resolving game. \textit{Bull. Malays. Math. Sci. Soc.} \textbf{44} (2021) 2081-2099.

\bibitem{landmarks} S. Khuller, B. Raghavachari and A. Rosenfeld, Landmarks in graphs. \textit{Discrete Appl. Math.} \textbf{70} (1996) 217-229.

\bibitem{chemistry} D.J. Klein and E. Yi, A comparison on metric dimension of graphs, line graphs, and line graphs of the subdivision graphs. \textit{Eur. J. Pure Appl. Math.} \textbf{5}(3) (2012) 302-316.

\bibitem{cop} R. Nowakawski and P. Winkler, Vertex-to-vertex pursuit in a graph. \textit{Discret. Math.} \textbf{43} (1983) 235-239.

\bibitem{sebo} A. Seb\"{o} and E. Tannier, On metric generators of graphs. \textit{Math. Oper. Res.} \textbf{29} (2004) 383-393.

\bibitem{slater} P.J. Slater, Leaves of trees. \textit{Congr. Numer.} \textbf{14} (1975) 549-559.

\bibitem{tilquist_thesis} R.C. Tillquist, Low-dimensional embeddings for symbolic data science. Ph.D. thesis, University of Colorado, Boulder, 2020.

\bibitem{tilquist} R.C. Tillquist, R.M. Frongillo and M.E. Lladser, Truncated metric dimension for finite graphs. arXiv:2106.14314v1 (2021) https://arxiv.org/abs/2106.14314

\bibitem{FMBRG} E. Yi, Fractional Maker-Breaker resolving game.  \textit{Lecture Notes in Comput. Sci.} \textbf{12577} (2020) 577-593.

\end{thebibliography}
\end{document}